\documentclass[final,leqno]{siamltex704}
\usepackage{amsmath}
\usepackage{graphicx}
\usepackage[notcite,notref]{showkeys}
\usepackage{mathrsfs}
\usepackage{appendix}
\usepackage{float}
\usepackage{amsfonts,amssymb}
\usepackage{dsfont}
\usepackage{pifont}
\usepackage{wrapfig} % Allows in-line images
\usepackage{hyperref}
\usepackage{multirow}
\numberwithin{equation}{section}
\usepackage{tikz}%% The amssymb package provides various useful mathematical symbols
\def\3bar{{|\hspace{-.02in}|\hspace{-.02in}|}}

\def\T{{\mathcal{T}}}

\def\bft{\textbf{t}}
\def\bfn{\textbf{n}}

\def\bn{{\mathbf{n}}}

\title{Superconvergence of Ritz-Galerkin Finite Element Approximations for Second Order Elliptic Problems}

\author{Chunmei Wang
\thanks{Department of Mathematics, Texas State University, San Marcos, TX 78666, USA. The research of Chunmei Wang was partially supported by National Science Foundation Awards DMS-1522586 and DMS-1648171.}}

\begin{document}

\maketitle

\begin{abstract}
In this paper, the author derives an $O(h^4)$-superconvergence for the piecewise linear Ritz-Galerkin finite element approximations for the second order elliptic equation $-\nabla \cdot(A\nabla u)= f$ equipped with Dirichlet boundary conditions. This superconvergence error estimate is established between the finite element solution and the usual Lagrange nodal point interpolation of the exact solution, and thus the superconvergence at the nodal points of each element. The result is based on a condition for the finite element partition characterized by the coefficient tensor $A$ and the usual shape functions on each element, called $A$-equilateral assumption in this paper. Several examples are presented for the coefficient tensor $A$ and finite element triangulations which satisfy the conditions necessary for superconvergence. Some numerical experiments are conducted to confirm this new theory of superconvergence.
\end{abstract}



\pagestyle{myheadings}

\section{Introduction}
Superconvergence is a phenomena in numerical methods that refer to faster than normal convergence for the approximate solutions arising from the numerical procedures. The research on superconvergence for finite element methods has been conducted extensively by many researchers over the last four decades. To the best of our knowledge, this phenomenon was first addressed in \cite{or1969}, and the term ``superconvergence'' was first used in \cite{dd1973}. Since then, superconvergence has become to be an active research topic in finite element methods for partial differential equations; see \cite{bs1997, cz2015, c2015, cw2003, elw1991, kn1987, ssw1996, w1991, w1995, w2000, z1978, zl1989, zz1992} and the references cited therein for an incomplete list of publications. An extensive bibliography on superconvergence was given in \cite{kn1998}, and many references for $3D$ problems can be found in \cite{k2005}.

In this paper, we are concerned with new developments of superconvergence for the classical piecewise linear Ritz-Galerkin finite element solutions of the second order elliptic equations. The model problem seeks an unknown function $u$ satisfying
\begin{equation}\label{a}
\begin{split}
-\nabla \cdot(A\nabla u)=& f,\qquad \text{in}\ \Omega, \\
u= &g, \qquad \text{on}\ \partial\Omega,
\end{split}
\end{equation}
where $\Omega \subset {\mathbb R}^2$ is an open bounded domain with Lipschitz continuous boundary $\partial \Omega$, and the coefficient tensor $A \in {\mathbb R}^{2 \times 2} $ is a symmetric, positive definite and constant matrix. The usual weak form for the model problem (\ref{a}) seeks $u \in H^1(\Omega)$ such that $u =g$ on $\partial\Omega$ and satisfying
\begin{equation}\label{weak}
(A\nabla u, \nabla v)=(f,v), \qquad \forall v\in H_0^1(\Omega).
\end{equation}
where $H^1(\Omega)$ is the Sobolev space on $\Omega$ consisting of $L^2$-functions with square-integrable first order partial derivatives, $H_0^1(\Omega)=\{v\in H^1(\Omega): v|_{\partial \Omega}=0\}$ is a closed subspace of $H^1(\Omega)$, and $(\cdot,\cdot)$ denotes the standard $L^2$ inner product in $L^2(\Omega)$.

The Ritz-Galerkin finite element method for (\ref{a}) is based on the weak formulation (\ref{weak}) by restricting the continuous Sobolev spaces into their subspaces consisting of $C^0$-piecewise polynomial finite element functions. In the classical theory for the Ritz-Galerkin finite element method, the optimal order of error estimate in $H^1$ for the finite element solution is bounded by $O(h)$ when linear elements are employed. This error estimate was well-known to be sharp. Any convergence with an order higher than $O(h)$ in the $H^1$ norm would be considered as superconvergence. The goal of this paper is to derive an $O(h^4)$-superconvergence error estimate for the finite element solution and the usual nodal point interpolation of the exact solution in $H^1$ and $L^{\infty}$ norms. This result shall be established for uniform finite element partitions consisting of a particular set of triangles known as $A$-equilateral triangles. Briefly speaking, a triangle $T$ is said to be $A$-equilateral if $(A\nabla \phi_i,\nabla\phi_i)_T=const$ for every shape function $\phi_i$ of the triangle $T$. For the identity matrix $A=I$, a triangle is $A$-equilateral if and only if it is equilateral in the conventional sense.

We follow the usual notation for Sobolev spaces and norms \cite{ciarlet-fem, Gilbarg-Trudinger, grisvard, gr}. For any open bounded domain $D\subset \mathbb{R}^d$ with
Lipschitz continuous boundary, we use $\|\cdot\|_{s,D}$ and
$|\cdot|_{s,D}$ to denote the norm and seminorms in the Sobolev
space $H^s(D)$ for any $s\ge 0$, respectively. The inner product
in $H^s(D)$ is denoted by $(\cdot,\cdot)_{s,D}$. The space
$H^0(D)$ coincides with $L^2(D)$, for which the norm and the inner
product are denoted by $\|\cdot \|_{D}$ and $(\cdot,\cdot)_{D}$,
respectively. When $D=\Omega$, we shall drop the subscript $D$ in
the norm and inner product notation.

This paper is organized as follows. In Section 2 we first review the classical Ritz-Galerkin finite element scheme for the model problem (\ref{a}) and then state the superconvergence result. Section 3 is devoted to a proof of the superconvergence error estimates in both the $H^1$ and $L^{\infty}$ norms. In Section 4, we shall discuss the invariance of $A$-equilateral triangles under translation and certain rotation and reflections; these properties are essential for the construction of finite element meshes on which  $O(h^4)$-superconvergence is possible. In Section 5, we report some numerical results to confirm the $O(h^4)$-superconvergence developed in Section 3. Finally in the Appendix section, we state and prove the well-known Euler-Maclaurin formula which plays an important role in the superconvergence analysis.

\section{The Ritz-Galerkin Finite Element Method and Superconvergence} In this section, we shall briefly review the classical Ritz-Galerkin finite element method for the second order elliptic problem (\ref{a}).

\begin{figure}[h] \vspace*{-1cm}\hspace*{-15pt}
\includegraphics[scale=0.6]{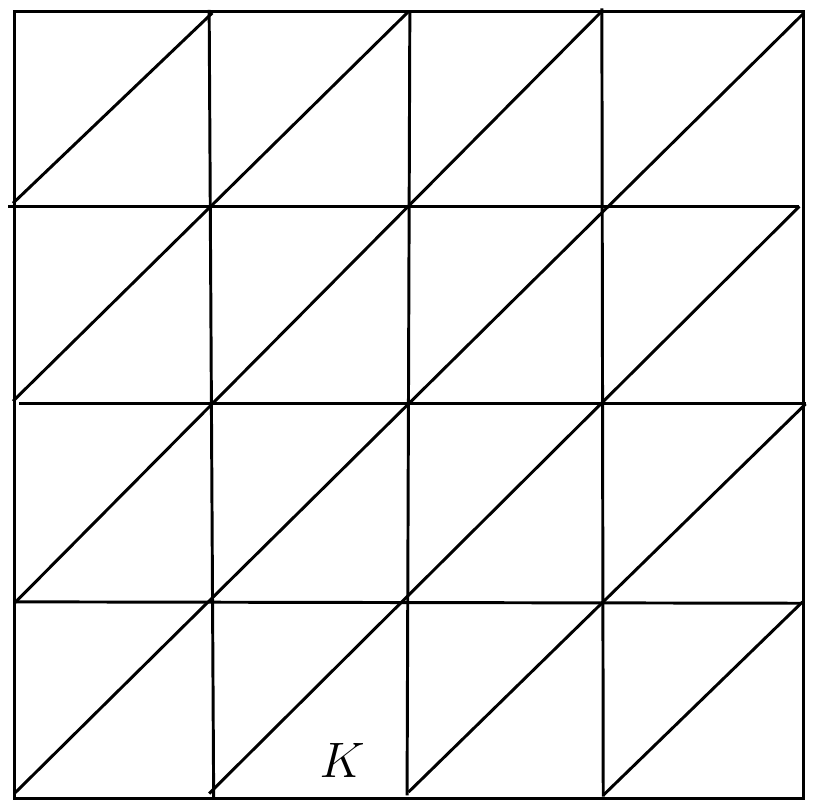}\vspace*{-10cm}
\caption{Finite Element Partition ${\cal T}_h$. }
\label{tri}
\end{figure}

Let ${\cal T}_h$ be a finite element partition of the $\Omega$ consisting of shape-regular triangles. Figure \ref{tri} illustrates a uniform finite element partition for a rectangular domain constructed as follows: First, the domain $\Omega$ is partitioned uniformly into $n\times n$ rectangles; Secondly, each rectangle is divided into two triangles by its diagonal line with a positive slope.

Denote by $S_h$ the finite element space consisting of $C^0$ piecewise linear functions; i.e.,
$$
S_h=\{ v: \ v \ \text{is continuous in}\ \Omega,\ v|_T\in P_1(T), T\in\T_h\},
$$
where $P_k(T)$ stands for the space of polynomial of total degree $k\ge 0$ or less. Denote by $S_h^0$ the subspace of $S_h$ consisting of finite element functions with vanishing boundary value; i.e.,
$$
S_h^0=\{ v: \ v \in S_h, \ v|_{\partial\Omega}=0\}.
$$
For any function $w\in H^1(\Omega)\cap C^0(\Omega)$, denote by $w^I\in S_h$ the usual interpolation of $w$ by using the Lagrange nodal basis.

The following is the well-known Ritz-Galerkin finite element scheme for the second order elliptic problem (\ref{a}) based on the weak form (\ref{weak}): {\em
Find $u_h \in S_h$ such that $u_h =g_I$ on $\partial\Omega$ and satisfying
\begin{equation}\label{fem}
(A\nabla u_h, \nabla v)=(f,v), \qquad \forall v\in S_h^0.
\end{equation}
}

Note that $S_h^0\subset H_0^1(\Omega)$. Thus, it follows from (\ref{fem}) and (\ref{weak}) that we have the following orthogonality
\begin{equation}\label{error-equation}
(A(u-\nabla u_h), \nabla v)=0, \qquad \forall v\in S_h^0.
\end{equation}
The above equation is also known as the error equation.

\medskip

For each triangular element $T\in \T_h$, denote by $\phi_{i,T}(i=1,2,3)$ the usual shape functions with value $1$ at one of the three vertices and $0$ at the other two. An element $T$ is said to be $A-$equilateral if there exists a constant $\alpha_T$ such that
\begin{equation}\label{A-equilateral}
(A\nabla\phi_{i,T},\nabla\phi_{i,T})_T=\alpha_T,\qquad i=1,2,3.
\end{equation}
The finite element partition $\T_h$ is said to be uniformly $A-$equilateral if there is a constant $\alpha$ such that
\begin{equation}\label{uniformly-A-equilateral}
(A\nabla\phi_{i,T},\nabla\phi_{i,T})_T=\alpha,\qquad i=1,2,3, \ T\in \T_h.
\end{equation}

The finite element partition $\T_h$ is said to be uniform if any two adjacent triangles that share a common edge form a parallelogram. The rest of this paper will assume that $\T_h$ is uniform and the coefficient tensor $A$ is a constant matrix.

The following is the main result of this paper regarding superconvergence for the finite element solution of the model problem (\ref{a}).

\begin{theorem} \label{thm1-new} Let $u_h$ be the Ritz-Galerkin finite element approximation arising from (\ref{fem}) and $u_I$ be the nodal point interpolation of the exact solution of (\ref{weak}). Assume that the finite element partition $\T_h$ is uniform and that the elements in $\T_h$ are uniformly $A-$equilateral. If the exact solution is sufficient smooth such that $u\in H^5(\Omega)$, then there exists a constant $C$ satisfying
\begin{equation}\label{h1norm-new}
\|\nabla(u_h - u_I)\|_0 \leq C h^4 \|u\|_{5}.
\end{equation}
Moreover, one has
\begin{equation}\label{maxnorm-new}
\|u_h - u_I\|_{\infty}\leq C h^4 |\log(h)|^{\frac12}\|u\|_5.
\end{equation}
\end{theorem}

The error estimate (\ref{maxnorm-new}) shows that the Ritz-Galerkin finite element solution is super-convergent to the exact solution at the rate of ${\cal O}(h^4|\log(h)|^{\frac12})$ when measured at the set of vertices in the maximum norm. Likewise, the error estimate (\ref{h1norm-new}) implies a superconvergence of order $4$ in a discrete $H^1$ norm for the finite element approximation $u_h$.

\section{A Proof for Theorem \ref{thm1-new}}

In this section we shall provide a proof for the superconvergence given by Theorem \ref{thm1-new}. To this end, for any triangle $T\in {\cal T}_h$ with vertices $V_1$, $V_2$ and $V_3$ ordered in the counterclockwise direction, denote by $|T|$ the area of the triangle $T$ and $\textbf{n}$ the unit outward normal direction to $\partial T$. Let $\varphi_j$ the shape function associated with the nodal point $V_j$ which assumes the value $1$ at the vertex point $V_j$ and $0$ at all other two vertices $V_k$ for $k\neq j (k, j=1,2,3)$. Denote by $\ell_{ij}$ the edge $V_iV_j$ as well as its length
($i,j=1,2,3$). Denote by $\textbf{n}_1$, $\textbf{n}_2$ and $\textbf{n}_3$ the unit outward normal directions to the edges $\ell_{23}$, $\ell_{31}$ and $\ell_{12}$, respectively, see Figure \ref{fig:triangle} for an illustration. Denote by $\textbf{t}_{12}$, $\textbf{t}_{23}$ and $\textbf{t}_{31}$ the unit tangential directions along the directions $\overrightarrow{V_1V_2}$, $\overrightarrow{V_2V_3}$ and $\overrightarrow{V_3V_1}$, respectively. For convenience, we use ${\cal D}_{12}u$, ${\cal D}_{23}u$ and ${\cal D}_{31}u$ to denote the partial derivatives along the directions $\mathbf{t}_{12}$, $\mathbf{t}_{23}$ and $\mathbf{t}_{31}$, respectively.

\begin{figure}
\begin{center}
\begin{tikzpicture}[rotate=233]
    %define the points of triangle
    %you can also use \coordinate to define these points, display in the following case
    \path (0,0) coordinate (A3);
    \path (3,0) coordinate (A1);
    \path (3.4,-0.2) coordinate (A11);
    \path  (1.92, 1.44) coordinate (M);
    \path (0,4) coordinate (A2);
    \path (-0.2,4.4) coordinate (A21);
    \path (1.0,1.0) coordinate (center);
    \path (1.5,0) coordinate (A1half);
    \path (0,2) coordinate (A2half);
    \path (1.5,2) coordinate (A3half);
    \path (2.0,0.5) coordinate (a);
    \path (A1half) ++(0,-0.6) coordinate (A1To);
    \path (A2half) ++(-0.6,0) coordinate (A2To);
    \path (A3half) ++(38:0.6cm) coordinate (A3To);
    \draw (A3) -- (A1) -- (A2) -- (A3);
    \draw [dashed] (A3) -- (M);
    %\filldraw[black] (A1) circle(0.1);
    %\filldraw[black] (A2) circle(0.1);
    %\filldraw[black] (A3) circle(0.1);
    %\filldraw[black] (A1half) circle(0.1);
    %\filldraw[black] (A2half) circle(0.1);
    %\filldraw[black] (A3half) circle(0.1);
    %\filldraw[black] (center) circle(0.11);
    \draw node[above] at (A3)(2.625, 0.5) {$V_3$};
    \draw node[below] at (A1) {$V_1$};
    \draw node[below] at (M) {$M$};
    %\draw node[below] at (A11) {$A_1$};
    \draw node[below] at (A2) {$V_2$};
    %\draw node at (center) {T};
    \draw node[right] at (A1half)(2.625, 0.5) {$\ell_{31}$};
    %\draw node[below] at (A1) {A(1)};
    \draw node[left] at (A2half) {$\ell_{23}$};
    \draw node[above] at (A3half) {$\ell_{12}$};
    %\draw node[below] at (center) {$P_{j}(T)$};

    \draw[->,thick] (A1half) -- (A1To) node[above]{$\mathbf{n}_2$};
    \draw[->,thick] (A2half) -- (A2To) node[right]{$\mathbf{n}_1$};
    \draw[->,thick] (A3half) -- (A3To) node[right]{$\mathbf{n}_3$};
    %\draw (2.625, 0.5) arc(135:180:0.725);
    %\draw node at (a) {$\alpha_{23}$};
\end{tikzpicture}
\end{center}
\caption{An illustrative triangular element}
\label{fig:triangle}
\end{figure}
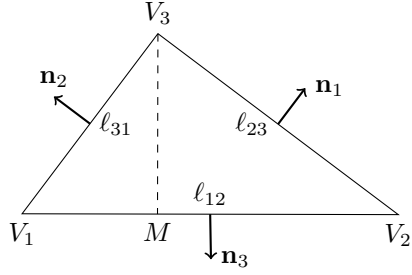

Let $\phi_{i,T}$ be the shape function corresponding to the vertex $V_i$. As functions increase the most rapidly along their gradient directions, we then have
\begin{equation}\label{bn-i}
\textbf{n}_i=-\nabla \phi_{i,T}/\|\nabla \phi_{i,T}\|,\qquad i=1,2,3,
\end{equation}
where $\|\nabla \phi_{i,T}\|$ is the Euclidean norm of the vector $\nabla \phi_{i,T}$. We claim that
\begin{equation}\label{3-new}
\|\nabla \phi_{3,T}\|^{-1}=2|T|/\ell_{12}.
\end{equation}
In fact, denote by $M\in \ell_{12}$ the endpoint of the perpendicular line segment passing through $V_3$ (see Figure \ref{fig:triangle}). Let $H$ be the length of the segment $MV_3$. From the definition of the shape function $\phi_{3,T}$ we have
$$
1= \phi_{3,T}(V_3)-\phi_{3,T}(M) = \overrightarrow{MV_3} \cdot \nabla  \phi_{3,T}= H\|\nabla \phi_{3,T}\|,
$$
which, together with $\ell_{12} H =2|T|$, gives rise to (\ref{3-new}). The relation (\ref{3-new}) can be extended to other two edges to give
\begin{equation}\label{3-new:002}
\|\nabla \phi_{1,T}\|^{-1}=2|T|/\ell_{23},\quad \|\nabla \phi_{2,T}\|^{-1}=2|T|/\ell_{31}.
\end{equation}
Substituting (\ref{3-new}) and (\ref{3-new:002}) into (\ref{bn-i}) yields the following result.

\begin{lemma} \label{Lemma:000} For any $T\in\T_h$ with vertices $V_i$ ordered in the counterclockwise direction, the following identities hold true:
\begin{equation}\label{bn:001}
\textbf{n}_1=-2|T|\nabla \phi_{1,T}/\ell_{23},\ \ \textbf{n}_2=-2|T|\nabla \phi_{2,T}/\ell_{31},\ \textbf{n}_3=-2|T|\nabla \phi_{3,T}/\ell_{12}.
\end{equation}
\end{lemma}

The following lemma provides three identities which are very useful in the superconvergence analysis.

\begin{lemma} \label{Lemma:001} For any $T\in\T_h$ and $w\in H^1(T)$ there hold the following identities:
\begin{equation}\label{2.9-1-new}
\int_{\ell_{12}} w ds=-\frac{\ell_{23}\ell_{12}}{2|T|}\int_T{\cal D}_{23}
w dT+\frac{\ell_{12}}{\ell_{31}}\int_{\ell_{31}}w ds.
\end{equation}
\begin{equation}\label{2.9-2-new}
\int_{\ell_{23}} w ds=-\frac{\ell_{31}\ell_{23}}{2|T|}\int_T{\cal D}_{31}
w dT+\frac{\ell_{23}}{\ell_{12}}\int_{\ell_{12}}w ds.
\end{equation}
\begin{equation}\label{2.9-3-new}
\int_{\ell_{31}} w ds=-\frac{\ell_{12}\ell_{31}}{2|T|}\int_T{\cal D}_{12}
w dT+\frac{\ell_{31}}{\ell_{23}}\int_{\ell_{23}}w ds.
\end{equation}
\end{lemma}
\begin{proof}
To prove (\ref{2.9-1-new}), from the usual integration by parts we have
\begin{equation}\label{2.6-1-new}
\begin{split}
 \int_T{\cal D}_{23} w dT
= \int_{\ell_{12}} w \textbf{n}_3\cdot\textbf{t}_{23}ds+\int_{\ell_{31}} w
\textbf{n}_2\cdot\textbf{t}_{23}ds.
\end{split}
\end{equation}
From the definition of the shape function $\phi_{3,T}$, it is not hard to see that
$$
\nabla\phi_{3,T}\cdot \bft_{23} = \ell_{23}^{-1}.
$$
Hence, we have from Lemma \ref{Lemma:000}
$$
\bft_{23}\cdot \bfn_3 = - 2|T| \nabla\phi_{3,T}\cdot \bft_{23} /\ell_{12}  = -\frac{2|T|}{\ell_{12}\ell_{23}}.
$$
Analogously, we have
$$
\bft_{23}\cdot \bfn_2 = - 2|T| \nabla\phi_{2,T}\cdot \bft_{23} /\ell_{31}  = \frac{2|T|}{\ell_{31}\ell_{23}}.
$$
Combining the last two identities with (\ref{2.6-1-new}) gives
\begin{equation*}
\begin{split}
\int_T {\cal D}_{23} w dT
=-\frac{2|T|}{\ell_{23}\ell_{12}}\int_{\ell_{12}} w ds + \frac{2|T|}{\ell_{23}\ell_{31}}\int_{\ell_{31}}w ds.
\end{split}
\end{equation*}
This completes the proof of (\ref{2.9-1-new}). The other two identities (\ref{2.9-2-new}) and (\ref{2.9-3-new}) can be proved in a similar fashion, and the details are omitted.
\end{proof}

Denote by $e_h=u_h-u_I$ the error between the Ritz-Galerkin finite element approximation $u_h$ and the nodal point interpolation $u_I$ of the exact solution $u$. From the error equation (\ref{error-equation}) we have
\begin{equation}\label{2.0}
(A\nabla e_h,\nabla v) = (A\nabla(u-u_I),\nabla v),\qquad \forall v\in S_h^0.
\end{equation}
Using the divergence theorem and the fact that $\nabla v$ is a constant vector on each element we obtain
\begin{equation}\label{2.1}
\begin{split}
&(A\nabla e_h,\nabla v)\\
=&\sum_{T\in {\cal T}_h}\int_T A\nabla(u-u_I)\cdot \nabla v dT\\
=&\sum_{T\in {\cal T}_h}\int_{\partial T}(u-u_I)(A\nabla v\cdot \textbf{n})ds\\
=&\sum_{T\in {\cal T}_h}\int_{\ell_{12}}(u-u_I)(A\nabla v\cdot \textbf{n}_3) ds
+\sum_{T\in {\cal T}_h}\int_{\ell_{23}}(u-u_I)(A\nabla v\cdot \textbf{n}_1) ds\\
& +
\sum_{T\in {\cal T}_h}\int_{\ell_{31}}(u-u_I)(A\nabla v\cdot \textbf{n}_2) ds\\
=& I_1 + I_2 + I_3,
\end{split}
\end{equation}
where $I_j$ are defined accordingly.

For simplicity of analysis, we shall focus on the treatment of the first term
\begin{equation}\label{I-one}
I_1=\sum_{T\in {\cal T}_h}\int_{\ell_{12}}(u-u_I)(A\nabla v\cdot \textbf{n}_3) ds
\end{equation}
in the forthcoming mathematical derivation; the other two terms can be handled by using the same method with minor and straightforward modifications.

Note that $v\in S_{h}^0$ can be represented by using the shape functions $\phi_{i,T}$ on each element $T\in \T_h$. For simplicity of notation, we shall drop the subscript $T$ from the notation of the basis functions. Thus, on the element $T$ we have
$$
v=v(V_1)\phi_1+v(V_2)\phi_2+v(V_3)\phi_3.
$$
It follows that
\begin{equation*}\label{1}
\nabla v=v(V_1)\nabla \phi_1+v(V_2)\nabla\phi_2+v(V_3)\nabla\phi_3,
\end{equation*}
which, together with $\nabla \phi_1+\nabla \phi_2+\nabla \phi_3=0$, gives
\begin{equation*}\label{aa}
\begin{split}
\nabla v&=(v(V_2)-v(V_1))\nabla \phi_2+(v(V_3)-v(V_1))\nabla\phi_3\\
&=\ell_{12}{\cal D}_{12}v\nabla \phi_2-\ell_{31}{\cal D}_{31}v\nabla \phi_3.
\end{split}
\end{equation*}
Substituting the above into (\ref{I-one}) yields
\begin{equation} \label{2.3-1}
\begin{split}
I_1=&\sum_{T\in {\cal T}_h} \int_{\ell_{12}}(u-u_I)(A\nabla v\cdot \textbf{n}_3)ds\\
=&\sum_{T\in {\cal T}_h}  \ell_{12}A \nabla\phi_2\cdot \bn_3 \int_{\ell_{12}}(u-u_I){\cal D}_{12}v ds \\& - \sum_{T\in {\cal T}_h} \ell_{31}A \nabla\phi_3 \cdot\bn_3 \int_{\ell_{12}}(u-u_I){\cal D}_{31}v ds\\
=&-\sum_{T\in {\cal T}_h}\ell_{31}A\nabla\phi_3 \cdot\bn_3 \int_{\ell_{12}}(u-u_I){\cal D}_{31}v ds.
\end{split}
\end{equation}
where we have used the fact that $\sum_{T\in {\cal T}_h}\ell_{12}A \nabla\phi_2\cdot \bn_3 \int_{\ell_{12}}(u-u_I){\cal D}_{12}v ds=0$ on the last line due to a similar contribution from its adjacent element which shares the same edge $\ell_{12}$ and hence forms a parallelogram with $T$, plus ${\cal D}_{12}v$ is continuous across $\ell_{12}$ and $\bn_3|_{T}=-\bn_3|_{T'}$; cf. Figure \ref{fig3}.

\begin{figure}[H] \vspace*{-2cm}\hspace*{-15pt}
\includegraphics[scale=0.5]{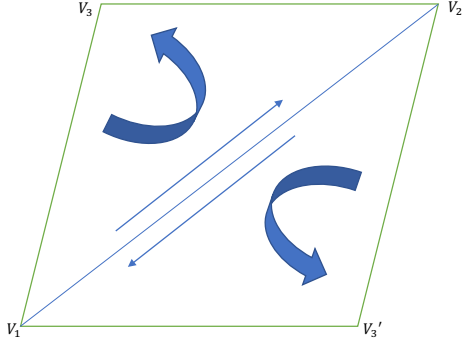}\vspace*{-7cm}
\caption{A parallelogram formed by two adjacent elements $T=\Delta V_1V_2V_3$ and $T'=\Delta V_1V_{3'}V_2$.}
 \label{fig3}
\end{figure}

Now substituting $\textbf{n}_3=-2|T|\nabla\phi_3/ \ell_{12}$ into (\ref{2.3-1}) gives
\begin{equation}\label{2.3}
I_1 = \sum_{T\in {\cal T}_h} \frac{2\ell_{31}|T|A \nabla\phi_3\cdot\nabla\phi_3}{\ell_{12}}\int_{\ell_{12}}(u-u_I){\cal D}_{31}v ds.
\end{equation}
Furthermore, we apply the Euler-Maclaurin formula (\ref{ML}) to the line integral $\int_{\ell_{12}}(u-u_I)ds$ to obtain
\begin{equation}\label{2.4}
\begin{split}
I_1=&\sum_{T\in {\cal T}_h}\int_{\ell_{12}}(u-u_I)(A\nabla v\cdot \textbf{n}_3)ds\\
=&\sum_{T\in {\cal T}_h}\frac{2 \ell_{31} |T|A\nabla \phi_3\cdot\nabla \phi_3 }{\ell_{12}}\int_{\ell_{12}}(u-u_I){\cal D}_{31}vds\\
=&\sum_{T\in {\cal T}_h}\frac{2 \ell_{31} |T|A\nabla \phi_3\cdot\nabla \phi_3 }{\ell_{12}}\Big( -\frac{ \ell_{12}^2}{12} \int_{\ell_{12}} {\cal D}_{12}^2 u {\cal D}_{31}v ds \\
&\quad +\left(\frac{\ell_{12}}{2}\right)^4\int_{\ell_{12}}{\cal G} {\cal D}_{12}^4 u {\cal D}_{31}v ds\Big).\\
\end{split}
\end{equation}
Using (\ref{2.9-1-new}), the line integral $\int_{\ell_{12}} {\cal D}_{12}^2 u {\cal D}_{31}v ds$ can be expressed as
\begin{equation}\label{2.4:100}
\int_{\ell_{12}} {\cal D}_{12}^2 u {\cal D}_{31}v ds = - \frac{\ell_{23}\ell_{12}}{2|T|} \int_T {\cal D}_{23} {\cal D}_{12}^2 u {\cal D}_{31}v dT + \frac{\ell_{12}}{\ell_{31}} \int_{\ell_{31}} {\cal D}_{12}^2 u {\cal D}_{31}v ds.
\end{equation}
To deal with the second line integral $\int_{\ell_{12}} {\cal G} {\cal D}_{12}^4 u {\cal D}_{31}v ds$, we shall extend the weight function ${\cal G}$ from the line segment $\ell_{12}$ to the element $T$ by assigning a constant value along the direction of $\ell_{23}$. Denote by ${\cal G}_{23}$ this extension of the weight function ${\cal G}$. Figure \ref{figfig2} illustrates how this extension was done geometrically: on each line segment $AB$ that is parallel to the edge $\ell_{23}=V_2V_3$, one sets ${\cal G}_{23}|_{{AB}}={\cal G}|_{\ell_{31}}(A)={\cal G}|_{\ell_{12}}(B)$. As ${\cal G}_{23}$ has constant values along the direction $V_2V_3$, we then have $D_{23}{\cal G}_{23}=0$.

\begin{figure}[H]
 \vspace*{-0.4cm}
\hspace*{40pt}
\includegraphics[scale=0.6]{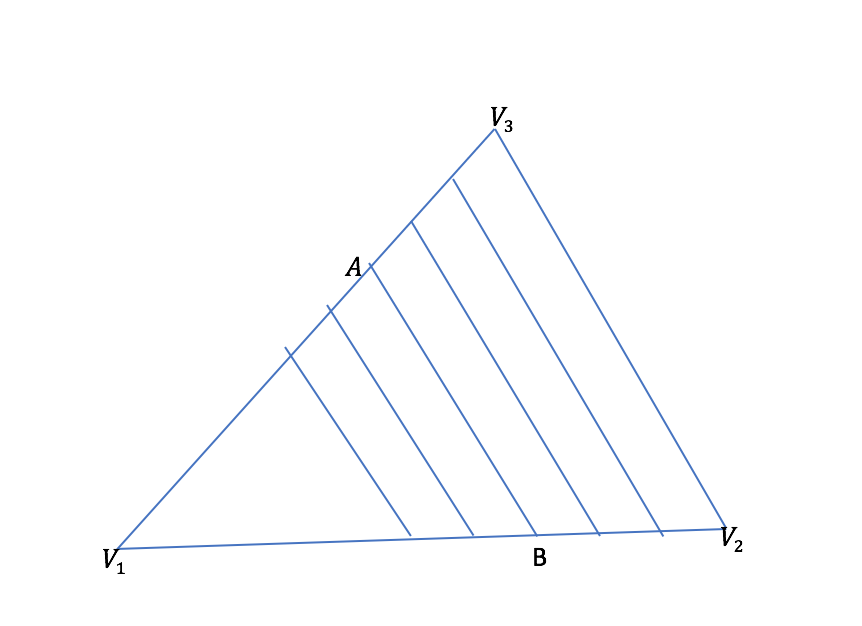}
%\vspace*{-10cm}
\caption{Extension of weight function ${\cal G}$ to the element $T=\Delta V_1V_2V_3$. }
\label{figfig2}
\end{figure}

Using the function ${\cal G}_{23}$, we may rewrite the line integral $\int_{\ell_{12}} {\cal G} {\cal D}_{12}^4 u {\cal D}_{31}v ds$ as follows:

\begin{equation}\label{2.4:110}
\int_{\ell_{12}}{\cal G} {\cal D}_{12}^4 u {\cal D}_{31}v ds = - \frac{\ell_{23}\ell_{12}}{2|T|} \int_T {\cal G}_{23} {\cal D}_{23} {\cal D}_{12}^4 u {\cal D}_{31}v dT + \frac{\ell_{12}}{\ell_{31}}  \int_{\ell_{31}} {\cal G}_{23} {\cal D}_{12}^4 u {\cal D}_{31}v ds.
\end{equation}
Now substituting (\ref{2.4:100}) and (\ref{2.4:110}) into (\ref{2.4}) yields,
 \begin{equation}\label{2.10}
 \begin{split}
I_1=&\sum_{T\in {\cal T}_h} \int_{\ell_{12}}(u-u_I)(A\nabla v\cdot \textbf{n}_3)ds\\
=&\sum_{T\in {\cal T}_h} \frac{ \ell^2_{12} \ell_{23} \ell_{31}A\nabla \phi_3\cdot\nabla \phi_3 }{12}\int_T{\cal D}_{23}{\cal D}_{12}^2u{\cal D}_{31}v dT
\\&-\sum_{T\in {\cal T}_h} \frac{\ell^2_{12} |T|A\nabla \phi_3\cdot\nabla \phi_3 }{6}\int_{\ell_{31}}{\cal D}_{12}^2u{\cal D}_{31}v ds\\
&-\sum_{T\in {\cal T}_h}\frac{ \ell^4_{12} \ell_{23} \ell_{31}  A\nabla \phi_3\cdot\nabla \phi_3 }{16} \int_T {\cal G}_{23}{\cal D}_{23}{\cal D}_{12}^4u{\cal D}_{31}v dT\\
&+ \sum_{T\in {\cal T}_h}\frac{\ell_{12}^4|T|A\nabla \phi_3\cdot\nabla \phi_3}{8}\int_{\ell_{31}} {\cal G} {\cal D}_{12}^4u  {\cal D}_{31}v ds \\
=&\sum_{T\in {\cal T}_h} \frac{\ell^2_{12} \ell_{23} \ell_{31}A\nabla \phi_3\cdot\nabla \phi_3 }{12}\int_T{\cal D}_{23}{\cal D}_{12}^2u{\cal D}_{31}v dT\\
& -\sum_{T\in {\cal T}_h}\frac{\ell^4_{12} \ell_{23}  \ell_{31}A\nabla \phi_3\cdot\nabla \phi_3 }{16} \int_T {\cal G}_{23} {\cal D}_{23}{\cal D}_{12}^4u{\cal D}_{31}v dT,\\
\end{split}
\end{equation}
where we have used the following cancellation property
$$
\sum_{T\in {\cal T}_h} \frac{\ell^2_{12} |T| A\nabla \phi_3\cdot\nabla \phi_3  }{6}\int_{\ell_{31}}{\cal D}_{12}^2u{\cal D}_{31}v ds=0,
$$
and
$$
\sum_{T\in {\cal T}_h}\frac{\ell_{12}^4|T|A\nabla \phi_3\cdot\nabla \phi_3}{8}\int_{\ell_{31}} {\cal G} {\cal D}_{12}^4u {\cal D}_{31}v ds=0,
$$
due to a similar contribution from its adjacent
element that shares the same edge $\ell_{31}$ and makes a parallelogram, plus the fact that ${\cal D}_{31} v $ has the same value on this parallelogram.

Similarly, we can derive the following identities:
\begin{equation}\label{2.11}
\begin{split}
I_2=&\sum_{T\in {\cal T}_h}\int_{\ell_{23}}(u-u_I)(A\nabla v\cdot \textbf{n}_1)ds\\
=&\sum_{T\in {\cal T}_h} \frac{\ell^2_{23}\ell_{31}\ell_{12} A\nabla \phi_1\cdot\nabla \phi_1}{12}\int_{T}{\cal D}_{31}
{\cal D}_{23}^2u{\cal D}_{12}v dT\\
&-\sum_{T\in {\cal T}_h}\frac{  \ell^4_{23} \ell_{31}  \ell_{12}A\nabla \phi_1\cdot\nabla \phi_1 }{16} \int_T {\cal G}_{31} {\cal D}_{31}
{\cal D}_{23}^4u {\cal D}_{12}v dT,
\end{split}
\end{equation}
and
\begin{equation}\label{2.12}
\begin{split}
I_3=&\sum_{T\in {\cal T}_h}\int_{\ell_{31}}(u-u_I)(A\nabla v\cdot\textbf{n}_2)ds\\
=&\sum_{T\in {\cal T}_h}\frac{\ell^2_{31} \ell_{12}\ell_{23}A\nabla \phi_2\cdot\nabla \phi_2}{12}\int_{T}
{\cal D}_{12}{\cal D}_{31}^2u{\cal D}_{23}v dT\\
&-\sum_{T\in {\cal T}_h}\frac{  \ell^4_{31} \ell_{12}  \ell_{23}A\nabla \phi_2\cdot\nabla \phi_2}{16} \int_T {\cal G}_{12} {\cal D}_{12}
{\cal D}_{31}^4u {\cal D}_{23}v dT,
\end{split}
\end{equation}
where ${\cal G}_{31}$ and ${\cal G}_{12}$ are similar extensions satisfying $D_{31}{\cal G}_{31}=0$ and $D_{12}{\cal G}_{12}=0$.

Going back to (\ref{2.1}), by using (\ref{2.10}), (\ref{2.11}) and (\ref{2.12}) we arrive at
\begin{equation}\label{2.13}
\begin{split}
&(A\nabla(u-u_I),\nabla v) \\
=&\frac{\ell_{12}\ell_{23}\ell_{31}}{12}\sum_{T\in{\cal T}_h}\int_{T}(\ell_{12}{\cal D}_{23}D^2_{12}u{\cal D}_{31}v
A \nabla\phi_3\cdot\nabla\phi_3\\
&+\ell_{23}{\cal D}_{31}D^2_{23}u{\cal D}_{12}v  A\nabla\phi_1\cdot\nabla\phi_1
+\ell_{31}{\cal D}_{12}D^2_{31}u{\cal D}_{23}v A \nabla\phi_2\cdot\nabla\phi_2)dT\\
& -\sum_{T\in {\cal T}_h}\frac{  \ell^4_{12} \ell_{23}  \ell_{31}A\nabla \phi_3\cdot\nabla \phi_3 }{16} \int_T {\cal G}_{23} {\cal D}_{23}
{\cal D}_{12}^4u {\cal D}_{31}v dT\\
&-\sum_{T\in {\cal T}_h}\frac{  \ell^4_{23} \ell_{31}  \ell_{12} A\nabla \phi_1\cdot\nabla \phi_1 }{16} \int_T {\cal G}_{31} {\cal D}_{31}
{\cal D}_{23}^4u {\cal D}_{12}vdT\\
&-\sum_{T\in {\cal T}_h}\frac{\ell^4_{31} \ell_{12}  \ell_{23} A\nabla \phi_2\cdot\nabla \phi_2}{16} \int_T {\cal G}_{12} {\cal D}_{12}
{\cal D}_{31}^4u {\cal D}_{23}v dT
\\
= &-\frac{\ell_{12}\ell_{23}\ell_{31}}{12}\sum_{T\in{\cal T}_h} \int_{T} {\cal D}_{12}{\cal D}_{23}{\cal D}_{31}
(A\nabla\phi_3\cdot\nabla\phi_3 \ell_{12}{\cal D}_{12}
+A\nabla\phi_1\cdot\nabla\phi_1 \ell_{23}{\cal D}_{23}\\
&+A\nabla\phi_2\cdot\nabla\phi_2 \ell_{31}{\cal D}_{31})u vdT + E_h,\\
\end{split}
\end{equation}
where
\begin{equation*}\label{2.13:100}
\begin{split}
E_h=& -\sum_{T\in {\cal T}_h}\frac{  \ell^4_{12} \ell_{23}  \ell_{31}A\nabla \phi_3\cdot\nabla \phi_3 }{16} \int_T {\cal G}_{23} {\cal D}_{23}
{\cal D}_{12}^4u {\cal D}_{31}v dT\\
&-\sum_{T\in {\cal T}_h}\frac{  \ell^4_{23} \ell_{31}  \ell_{12} A\nabla \phi_1\cdot\nabla \phi_1 }{16} \int_T {\cal G}_{31} {\cal D}_{31}
{\cal D}_{23}^4u {\cal D}_{12}vdT\\
&-\sum_{T\in {\cal T}_h}\frac{\ell^4_{31} \ell_{12}  \ell_{23} A\nabla \phi_2\cdot\nabla \phi_2}{16} \int_T {\cal G}_{12} {\cal D}_{12}
{\cal D}_{31}^4u {\cal D}_{23}v dT.
\end{split}
\end{equation*}
As $\T_h$ is uniformly $A$-equilateral, then there is a constant $\alpha$ such that $A\nabla \phi_i\cdot \nabla\phi_i=\alpha$ for $i=1,2,3$. It follows that
\begin{equation*}
\begin{split}
& A\nabla\phi_3\cdot\nabla\phi_3 \ell_{12}{\cal D}_{12}
+A\nabla\phi_1\cdot\nabla\phi_1 \ell_{23}{\cal D}_{23}
+A\nabla\phi_2\cdot\nabla\phi_2 \ell_{31}{\cal D}_{31})u \\
= & \alpha (\ell_{12}{\cal D}_{12}
+ \ell_{23}{\cal D}_{23}
+ \ell_{31}{\cal D}_{31})u \\
= & 0.
\end{split}
\end{equation*}
Substituting the above into (\ref{2.13}) yields
\begin{equation*}
(A\nabla(u-u_I),\nabla v) = E_h.
\end{equation*}
Finally, since $|\nabla \phi_i|=O(\frac{1}{h})$ for $i=1, 2, 3$, then we have
\begin{equation}\label{2.13:200}
|E_h| \leq C h^4 \|u\|_5 \|\nabla v\|_0.
\end{equation}
Combining the last two gives rise to the following estimate
\begin{equation}\label{2.18:100}
|(A\nabla(u_h-u_I),\nabla v)| = |(A\nabla(u-u_I),\nabla v)| \leq C h^4 \|u\|_5 \|\nabla v\|_0
\end{equation}
for all $v\in S_0^h$. In particular, by setting $v=u_h-u_I$ we arrive at
$$
|(A\nabla(u_h-u_I),\nabla (u_h-u_I))| \leq C h^4 \|u\|_5 \|\nabla (u_h-u_I)\|_0,
$$
which implies the superconvergence error estimate (\ref{h1norm-new}).

Since $u_h-u_I=0$ on $\partial \Omega$, there holds
$$
\|u_h-u_I\|_{\infty} \leq C |\log(h)|^{\frac12 }\|\nabla(u_h-u_I)\|_0.
$$
It follows that the superconvergence estimate (\ref{maxnorm-new}) in the maximum norm holds true.

\section{On Sufficient Conditions for Superconvergence} The proof for Theorem \ref{thm1-new}, particularly the identify (\ref{2.13}), provides the following sufficient condition for superconvergence
\begin{equation}\label{SuperC-Condition:01}
A\nabla\phi_3\cdot\nabla\phi_3 \ell_{12}{\cal D}_{12}
+A\nabla\phi_1\cdot\nabla\phi_1 \ell_{23}{\cal D}_{23}
+A\nabla\phi_2\cdot\nabla\phi_2 \ell_{31}{\cal D}_{31}=0.
\end{equation}
The condition (\ref{SuperC-Condition:01}) is satisfied if the triangular elements are $A$-equilateral; or equivalently if $A\nabla\phi_i\cdot\nabla\phi_i = \alpha,\ i=1,2,3,$ for a fixed real number $\alpha$. The goal of this section is to present some examples of the finite element partitions that are $A$-equilateral, and thus superconvergence can be concluded for the corresponding Ritz-Galerkin finite element solutions.

\subsection{Example 1} Our first example is concerned with the reference element $\widehat{T}$ with vertices $\widehat{V}_1=(0,0)$, $\widehat{V}_2=(1,0)$, and $\widehat{V}_3=(1,1)$. The three shape functions for this reference element are given by
\begin{equation*}
\begin{split}
{\widehat\phi}_1 & = 1-\widehat{x},\\
{\widehat\phi}_2 & =\widehat{x}-\widehat{y}, \\
{\widehat\phi}_3 & =\widehat{y}.
\end{split}
\end{equation*}
For the coefficient matrix
$$
\widehat{A}=\left[
 \begin{array}{cc}
 a_{11} & a_{12} \\
 a_{12} & a_{22} \\
 \end{array}
 \right],
$$
it can be easily calculated that
\begin{equation*}
\begin{split}
\widehat{A}\widehat{\nabla} \widehat\phi_1 \cdot\widehat\nabla\widehat\phi_1 & = a_{11},\\
\widehat A\widehat\nabla \widehat\phi_3 \cdot\widehat\nabla\widehat\phi_3 & = a_{22},\\
\widehat A\widehat\nabla \widehat\phi_2 \cdot\widehat\nabla\widehat\phi_2 & = a_{11}+a_{22} - 2 a_{12}.
\end{split}
\end{equation*}
Thus, the reference element $\widehat{T}$ is $\widehat{A}$-equilateral if and only if
$$
a_{11}=a_{22} = a_{11}+a_{22} - 2 a_{12},
$$
or equivalently,
$$
a_{11}=a_{22}=\alpha, a_{12}=\alpha/2.
$$
The coefficient matrix $\widehat{A}$ is thus given by
$$
\widehat{A}=\alpha \left[
  \begin{array}{cc}
 1 & 0.5 \\
 0.5 & 1 \\
 \end{array}
 \right].
$$

\subsection{Example 2} Our second example is concerned with an element $T$ with vertices
$V_1=(0,0)$, $V_2=(1,0)$, and $V_3=(x_0,y_0)$, where $y_0\neq 0$. We claim that $T$ is $A$-equilateral if and only if
\begin{equation}\label{MyEstimate:001}
A=S\widehat{A}_\alpha S^T,
\end{equation}
where
$$
S=\left[%
\begin{array}{cc}
 1 & x_0-1 \\
 0 & y_0\\
\end{array}%
\right]
$$
and
\begin{equation}\label{A-hat-alpha}
 \widehat{A}_\alpha=\alpha \left[
 \begin{array}{cc}
 1 & 0.5 \\
 0.5 & 1 \\
 \end{array}
 \right].
\end{equation}
In fact, the element $T$ is linked to the reference element $\widehat{T}$ through the following linear map:
$$
\left[
\begin{array}{c}
 x \\
 y\\
\end{array}
\right]
=
\left[%
\begin{array}{cc}
 1 & x_0-1 \\
 0 & y_0\\
\end{array}%
\right]
\left[
\begin{array}{c}
 \widehat{x} \\
 \widehat{y}\\
\end{array}
\right].
$$
A straightforward calculation shows that $\nabla\phi_i = S^{-T} \widehat\nabla\widehat\phi_i$. It follows that
\begin{eqnarray*}
A\nabla\phi_i\cdot \nabla\phi_i &=& A S^{-T} \widehat\nabla\widehat\phi_i \cdot S^{-T} \widehat\nabla\widehat\phi_i\\
&=& S^{-1} A S^{-T} \widehat\nabla\widehat\phi_i \cdot \widehat\nabla\widehat\phi_i.
\end{eqnarray*}
Thus, from {\bf Example 1}, the triangle $T$ is $A$-equilateral if and only if
$$
S^{-1} A S^{-T} = \widehat{A}_\alpha,
$$
which is equivalent to (\ref{MyEstimate:001}).

\subsection{Example 3} In the third example, the triangular element $T$ has a generic position with vertices $V_i=(x_i,y_i),\ i=1,2,3$.  We claim that $T$ is $A$-equilateral if and only if
\begin{equation}\label{MyEstimate:005}
A=S\widehat{A}_\alpha S^T,
\end{equation}
where
\begin{equation}\label{MyTerm:S}
S=\left[%
\begin{array}{cc}
 x_2-x_1 & x_3-x_2 \\
 y_2-y_1 & y_3-y_2 \\
\end{array}%
\right]
\end{equation}
and
\begin{equation}\label{A-hat-alpha-new}
 \widehat{A}_\alpha=\alpha \left[
 \begin{array}{cc}
 1 & 0.5 \\
 0.5 & 1 \\
 \end{array}
 \right].
\end{equation}
Note that the element $T$ can be transformed to the reference element $\widehat{T}$ through the following affine map:
$$
\left[
\begin{array}{c}
 x -x_1\\
 y -y_1\\
\end{array}
\right]
=
\left[%
\begin{array}{cc}
 x_2-x_1 & x_3-x_2 \\
 y_2-y_1 & y_3-y_2 \\
\end{array}%
\right]
\left[
\begin{array}{c}
 \widehat{x} \\
 \widehat{y}\\
\end{array}
\right].
$$
A straightforward calculation shows that $\nabla\phi_i = S^{-T} \widehat\nabla\widehat\phi_i$. Thus, we have
\begin{eqnarray*}
A\nabla\phi_i\cdot \nabla\phi_i &=& A S^{-T} \widehat\nabla\widehat\phi_i \cdot S^{-T} \widehat\nabla\widehat\phi_i\\
&=& S^{-1} A S^{-T} \widehat\nabla\widehat\phi_i \cdot \widehat\nabla\widehat\phi_i.
\end{eqnarray*}
From {\bf Example 1}, the triangle $T$ is $A$-equilateral if and only if
$$
S^{-1} A S^{-T} = \widehat{A}_\alpha,
$$
which gives rise to (\ref{MyEstimate:005}).

\subsection{Invariance of $A$-equilateral elements}

Let $T$ be a triangular element with vertices $V_i=(x_i,y_i),\ i=1,2,3$. $T'$ is said to be a translation of $T$ if there exists a point $(x^*, y^*)$ such that $T'$ is given by the set $(x,y)+(x^*,y^*)$ for all $(x,y)\in T$. This translation shall be denoted as $T'=T+(x^*,y^*)$. For example, in Figure \ref{fig:rotations} the triangular element $T'=\Delta V'_1V'_2V'_3$ is a translation of the reference triangle $T=\Delta V_1V_2V_3$ as $V'_i=V_i+(1.25, -1)$ for $i=1, 2,3$.

\begin{lemma} (translation invariance) If $T$ is $A$-equilateral with $A\nabla \phi_i\cdot\nabla \phi_i=\alpha$ and $T'$ is a translation of $T$, then $T'$ is also $A$-equilateral with the same $\alpha$.
\end{lemma}
\begin{proof}
The Example 3 shows that $T$ is $A$-equilateral with $A\nabla \phi_i\cdot\nabla \phi_i=\alpha$ if and only if $A = S \widehat{A}_\alpha S^T,$
where $S$ and $\widehat{A}_\alpha$ are given by (\ref{MyTerm:S}) and (\ref{A-hat-alpha-new}). As the matrix $S$ is translation invariant and $\widehat{A}_\alpha$ is fixed, the same representation then holds true on $T'$. It follows that $T'$ is $A$-equilateral with the same value $\alpha$.
\end{proof}

\begin{figure}
\begin{center}
\begin{tikzpicture}[rotate=0]
    %define the points of triangle
    %you can also use \coordinate to define these points, display in the following case
    \path (0,0) coordinate (A1);
    \path (2,0) coordinate (A2);
    \path (2,2) coordinate (A3);
    \path (0,2) coordinate (A4);
    \path (0,-2) coordinate (A5);
    \path (-2,-2) coordinate (A6);
    \path (-2,0) coordinate (A7);

    \path (2.5,-2) coordinate (A1P);
    \path (4.5,-2) coordinate (A2P);
    \path (4.5,0) coordinate (A3P);

    \path (5,0) coordinate (XR);
    \path (-5,0) coordinate (XL);
    \path (0,3) coordinate (YU);
    \path (0,-3) coordinate (YD);
    \path (-0.06,0.2) coordinate (A1M);

   \draw[->,thick] (XL) -- (XR);
   \draw[->,thick] (YD) -- (YU);

   \draw node[left] at (A1M) {$V_1$};
    \draw node[below] at (A2) {$V_2$};
    \draw node[right] at (A3) {$V_3$};
    \draw node[left] at (A4) {$V_4$};
    \draw node[right] at (A5) {$V_5$};
    \draw node[left] at (A6) {$V_6$};
    \draw node[above] at (A7) {$V_7$};
    \draw node[right] at (XR) {$X$};
    \draw node[above] at (YU) {$Y$};

   % \path (3.4,-0.2) coordinate (A11);
    %\path  (1.92, 1.44) coordinate (M);

    %\path (-0.2,4.4) coordinate (A21);
    %\path (1.0,1.0) coordinate (center);
    %\path (1.5,0) coordinate (A1half);
   % \path (0,2) coordinate (A2half);
    %\path (1.5,2) coordinate (A3half);
    %\path (2.0,0.5) coordinate (a);
    %\path (A1half) ++(0,-0.6) coordinate (A1To);
    %\path (A2half) ++(-0.6,0) coordinate (A2To);
    %\path (A3half) ++(38:0.6cm) coordinate (A3To);

    \definecolor{c1}{RGB}{0,129,188}
    \definecolor{c2}{RGB}{252,177,49}
    \definecolor{c3}{RGB}{35,34,35}
    \draw (A1) -- (A2) -- (A3) -- cycle;
    \shade [left color=c1,right color=c2] (A1) -- (A2) -- (A3) -- cycle;

    \draw (A1) -- (A4) -- (A3) -- cycle;
    \shade [left color=c2,right color=c1] (A1) -- (A4) -- (A3) -- cycle;

    \draw (A1) -- (A5) -- (A6) -- cycle;
    \shade [left color=c1,right color=c2] (A1) -- (A5) -- (A6) -- cycle;

    \draw (A1) -- (A6) -- (A7) -- cycle;
    \shade [left color=c2,right color=c1] (A1) -- (A6) -- (A7) -- cycle;

\draw (A1P) -- (A2P) -- (A3P) -- cycle;
    \shade [left color=c1,right color=c2] (A1P) -- (A2P) -- (A3P) -- cycle;
    \draw node[below] at (A1P) {$V'_1$};
    \draw node[below] at (A2P) {$V'_2$};
    \draw node[above] at (A3P) {$V'_3$};

    %\draw [dashed] (A3) -- (M);
    %\filldraw[black] (A1) circle(0.1);
    %\filldraw[black] (A2) circle(0.1);
    %\filldraw[black] (A3) circle(0.1);
    %\filldraw[black] (A1half) circle(0.1);
    %\filldraw[black] (A2half) circle(0.1);
    %\filldraw[black] (A3half) circle(0.1);
    %\filldraw[black] (center) circle(0.11);

    %\draw node[below] at (A11) {$A_1$};
    %\draw node[below] at (A2) {$V_2$};
    %\draw node at (center) {T};
    %\draw node[right] at (A1half)(2.625, 0.5) {$\ell_{31}$};
    %\draw node[below] at (A1) {A(1)};
    %\draw node[left] at (A2half) {$\ell_{23}$};
    %\draw node[above] at (A3half) {$\ell_{12}$};
    %\draw node[below] at (center) {$P_{j}(T)$};

    %\draw[->,thick] (A1half) -- (A1To) node[above]{$\mathbf{n}_2$};
    %\draw[->,thick] (A2half) -- (A2To) node[right]{$\mathbf{n}_1$};
    %\draw[->,thick] (A3half) -- (A3To) node[right]{$\mathbf{n}_3$};
    %\draw (2.625, 0.5) arc(135:180:0.725);
    %\draw node at (a) {$\alpha_{23}$};
\end{tikzpicture}
\end{center}
\caption{$A$-equilateral triangles generated by the reference triangle $\widehat{T}=\Delta V_1V_2V_3$ through translation, rotation and reflections.}
\label{fig:rotations}
\end{figure}
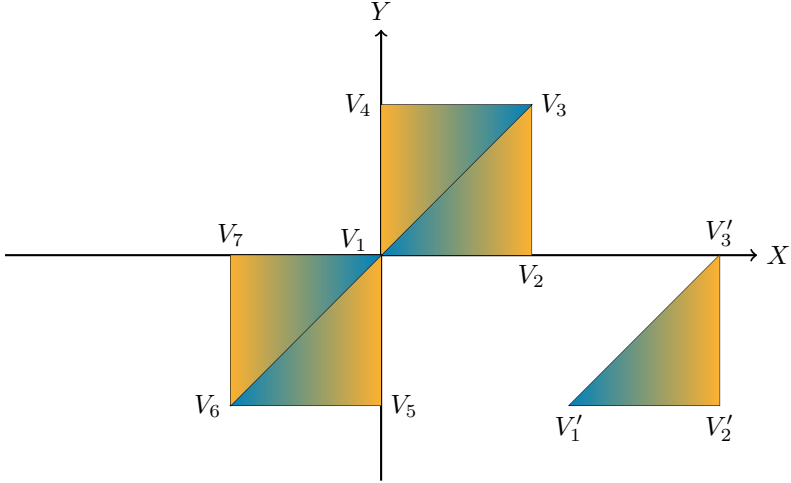

Assume that the element $T$ is $A$-equilateral with value $A\nabla \phi_i\cdot\nabla \phi_i=\alpha$. We would like to know among all the triangles that share vertex $V_1$ with $T$, which are also $A$-equilateral with the same value of $\alpha$. For simplicity, we shall consider the case of $T=\widehat{T}$ with $\alpha=2$. From Example 1, the matrix $\widehat{A}$ must be given by
$$
\widehat{A} = \left[
  \begin{array}{cc}
 2 & 1 \\
 1 & 2 \\
 \end{array}
 \right].
$$
Let $\tilde T$ be an arbitrary triangle that shares $\tilde V_1=(0,0)$ with the element $\widehat{T}$. Without loss of generality, we may assume the other two vertices of $\tilde T$ are given by $\tilde V_2=(\tilde x_2,\tilde y_2)$ and $\tilde V_3=(\tilde x_3,\tilde y_3)$. If $\tilde T$ is also $\widehat{A}$-equilateral with value $\alpha=2$, then from Example 3 we must have
\begin{equation}\label{Interesting:001}
\widehat{A}=\tilde A = \tilde S \widehat{A} \tilde S^T,
\end{equation}
where
\begin{equation}\label{MyTerm:STilde}
\tilde S=\left[%
\begin{array}{cc}
 \tilde x_2 & \tilde x_3-\tilde x_2 \\
 \tilde y_2 & \tilde y_3-\tilde y_2 \\
\end{array}%
\right].
\end{equation}
It follows from (\ref{Interesting:001}) that $(\det(\tilde{S}))^2=1$. Furthermore, a tedious calculation can be performed to show that the matrix $\tilde{S}$ can only take the following values:
$$
\tilde S=\left[%
\begin{array}{cc}
  1 & 0 \\
  0 & 1 \\
\end{array}%
\right],\
\left[\begin{array}{cc}
  -1 & 0 \\
  0 & -1 \\
\end{array}%
\right],\
\left[\begin{array}{cc}
  0 & 1 \\
  1 & 0 \\
\end{array}%
\right],\
\left[\begin{array}{cc}
  0 & -1 \\
  -1 & 0 \\
\end{array}%
\right].
$$
As illustrated in Figure \ref{fig:rotations}, the first value of $\tilde S$ corresponds to the original triangle $\widehat{T}=\Delta V_1V_2V_3$, the second gives the triangle $\Delta V_1V_7V_6$, the third one yields $\Delta V_1V_3V_4$, and the last one gives us the triangle $\Delta V_1V_6V_5$. A finite element partition with the superconvergence as described in Theorem \ref{thm1-new} must be formed by any subset of the elements in Figure \ref{fig:rotations} through translations with various values of $(x^*, y^*)$ that gives a valid computational partition of the domain.

\section{Numerical Tests}
In this section, we report some numerical results that confirm the superconvergence established in Theorem \ref{thm1-new} for the Ritz-Galerkin finite element solutions of the second order elliptic model problem (\ref{a}). The Ritz-Galerkin finite element method was implemented on uniform finite element partitions consisting of uniformly $A$-equilateral triangles. In our numerical tests, the exact solutions are taken as $u= \sin(x)\sin(y)$ and $u=\cos(x)\cos(y)$, respectively. The right-hand side function $f$ is computed to match the exact solutions.

Our first numerical example was conducted on the unit square domain $\Omega=(0,1)^2$. The model problem has the coefficient tensor $A=\left(
  \begin{array}{cc}
  2 & 1 \\
  1 & 2\\
  \end{array}
  \right)
$. The finite element partition was constructed so that it is uniformly $A$-equilateral.  Tables \ref{a1}-\ref{a2} contains the error information plus rate of convergence in the $L^2$ norm, $H^1$ semi-norm and $L^\infty$ norm. It is clear that the convergence in all three norms are of order $O(h^4)$ for both the exact solutions $u=\sin(x) \sin(y)$ and $u= \cos(x)\cos(y)$. The numerical results are consistent with the theory established in this paper.

\begin{table}[h!]
\begin{center}
 \caption{ Numerical error and convergence rates in the $L^2$ norm, $H^1$ semi-norm and $L^\infty$ norm with exact solution $u=\sin(x) \sin(y)$ and coefficient tensor $A=[2,1;1,2]$.} \label{a1}\begin{tabular}{|c|c|c|c|c|c|c|}
\hline
$1/h$ & $\|u_h-u_I\|_0$ & order & $\|\nabla(u_h-u_I)\|_0$ & order & $\|u_h-u_I\|_{\infty}$ & order\\
\hline
$2$ & 3.2847e-005 & & 1.3139e-004 & & 4.6453e-005 & \\
\hline
$4$ & 2.1222e-006 & 3.9521 & 1.0867e-005 & 3.5958 & 2.7046e-006 & 4.1023 \\
\hline
$8$ & 1.3194e-007 & 4.0076 & 7.2445e-007 & 3.9069 & 1.7945e-007 & 3.9138 \\
\hline
$16$ & 8.23041e-009 & 4.0027 & 4.6022e-008 & 3.9765 & 1.1186e-008 & 4.0039\\
\hline
$32$ & 5.1417e-010 & 4.0006 & 2.8883e-009 & 3.9940 & 7.0079e-010 & 3.9965 \\
\hline
$64$ & 3.2149e-011 & 3.9994 & 1.8078e-010 & 3.9979 & 4.3828e-011 & 3.9991\\
\hline
\end{tabular}
\end{center}
\end{table}

\begin{table}[h!]
\begin{center}
 \caption{ Numerical error and convergence rates in the $L^2$ norm, $H^1$ semi-norm and $L^\infty$ norm with exact solution $u=\cos(x)\cos(y)$ and coefficient tensor $A=[2,1;1,2]$.} \label{a2} \begin{tabular}{|c|c|c|c|c|c|c|}
\hline
$1/h$ & $\|u_h-u_I\|_0$ & order & $\|\nabla(u_h-u_I)\|_0$ & order & $\|u_h-u_I\|_{\infty}$ & order\\
\hline
$2$ & 3.5382e-005 & & 1.4153e-004 & & 5.0037e-005 & \\
\hline
$4$ & 2.2782e-006 & 3.9571 & 1.1620e-005 & 3.6065 & 2.9128e-006 & 4.1025\\
\hline
$8 $ & 1.4162e-007 & 4.0078 & 7.7330e-007 & 3.9093
& 1.9125e-007 &3.9288 \\
\hline
$16$ & 8.8348e-009 & 4.0027 & 4.9105e-008 & 3.9771 & 1.1922e-008 &4.0038 \\
\hline
$32$ & 5.5192e-010 & 4.0007 & 3.0814e-009 & 3.9942 & 7.4860e-010 &3.9933 \\
\hline
$64$ & 3.4433e-011 & 4.0025 & 1.9247e-010 & 4.0009 & 4.6708e-011 &4.0025 \\
\hline
\end{tabular}
\end{center}
\end{table}

The domain in the second test case is a parallelogram with vertices
$$
(0,0), \ (1.1462,0.9042), \ (0.6941, 2.2924),\ (-0.4521,1.3882).
$$
The coefficient tensor is given by $A=\left(
 \begin{array}{cc}
 2 & 2 \\
 2 & 8 \\
 \end{array}
 \right)
$, and the finite element partitions again consist of only $A$-equilateral triangles as required in Theorem \ref{thm1-new}.
Tables \ref{a3}-\ref{a4} illustrate the numerical performance with rate of convergence computed in various Sobolev norms. It can be seen that the numerical results confirm the theoretical predictions developed in the previous sections.

\begin{table}[h!]
\begin{center}
 \caption{ Numerical error and convergence rates in the $L^2$ norm, $H^1$ semi-norm and $L^\infty$ norm with exact solution $u=\sin(x)\sin(y)$ and coefficient tensor $A=[2,2;2,8]$.}
\label{a3}
 \begin{tabular}{|c|c|c|c|c|c|c|}
\hline
$1/h$ & $\|u_h-u_I\|_0$ & order & $\|\nabla(u_h-u_I)\|_0$ & order & $\|u_h-u_I\|_{\infty}$ & order\\
\hline
$2$ & 2.3693e-005 & & 9.4770e-005 & & 3.3506e-005 & \\
\hline
$4$ &8.4839e-006 & 1.4816 & 6.3017e-005 & 0.5887 & 1.3099e-005 & 1.3550\\
\hline
$8$ & 5.3785e-007 & 3.9794 & 4.6241e-006 & 3.7685 & 7.8251e-007 & 4.0652\\
\hline
$16$ & 3.3512e-008 & 4.0044 & 3.0020e-007 & 3.9452 & 4.9613e-008 & 3.9793\\
\hline
$32$ & 2.0923e-009 & 4.0015 & 1.8943e-008 & 3.9862 & 3.1083e-009 & 3.9965
\\
\hline
$64 $ & 1.3074e-010 & 4.0004 & 1.1868e-009 & 3.9965 & 1.9455e-010 & 3.9979\\
\hline
\end{tabular}
\end{center}
\end{table}

\begin{table}[h!]
\begin{center}
 \caption{ Numerical error and convergence rates in the $L^2$ norm, $H^1$ semi-norm and $L^\infty$ norm with exact solution $u=\cos(x)\cos(y)$ and coefficient tensor $A=[2,2;2,8]$.}\label{a4} \begin{tabular}{|c|c|c|c|c|c|c|}
\hline
$1/h$ & $\|u_h-u_I\|_0$ & order & $\|\nabla(u_h-u_I)\|_0$ & order & $\|u_h-u_I\|_{\infty}$ & order\\
\hline
$2 $ & 7.5777-005 & & 3.0311e-004 &
& 1.0716e-004 & \\
\hline
$4$ & 9.6454e-006 & 2.9738& 6.7760e-005 & 2.1613 & 1.6506e-005 &2.6987\\
\hline
$8$ & 6.0766e-007 & 3.9885& 4.9168e-006 & 3.7847 & 9.9509e-007 &4.0520\\
\hline
$16$ & 3.7857e-008 & 4.0046 & 3.1850e-007 & 3.9483 & 6.3017e-008 &3.9810\\
\hline
$32$ & 2.3637e-009 & 4.0014& 2.0088e-008 & 3.9869 & 3.9341e-009 &4.0017\\
\hline
$64$ &1.4770e-010 & 4.0003 & 1.2584e-009 & 3.9967 & 2.4582e-010 & 4.0003 \\
\hline
\end{tabular}
\end{center}
\end{table}

Our last numerical test was conducted on another parallelogram domain with vertexes
$$
(0,0), \ (0.7917,0.7672), \ (1.1238, 1.8184),\ (0.3322, 1.0512).
$$
The coefficient tensor is given by $A=\left(
 \begin{array}{cc}
 2 & 3 \\
 3 & 5 \\
 \end{array}
 \right)
$, and the finite element partition can be constructed to satisfy the uniform $A$-equilateral property as required by Theorem \ref{thm1-new}. Tables \ref{a5}-\ref{a6} show that the convergence rates in various norms. The numerical results are very much in consistency with the superconvergence developed in Theorem \ref{thm1-new}.

\begin{table}[h!]
\begin{center}
\caption{ Numerical error and convergence rates in the $L^2$ norm, $H^1$ semi-norm and $L^\infty$ norm with exact solution $u=\sin(x)\sin(y)$ and coefficient tensor $A=[2,3;3,5]$.} \label{a5}\begin{tabular}{|c|c|c|c|c|c|c|}
\hline
$1/h$ & $\|u_h-u_I\|_0$ & order & $\|\nabla(u_h-u_I)\|_0$ & order & $\|u_h-u_I\|_{\infty}$ & order\\
\hline
$2 $ & 6.1090e-005 & &2.4436e-004 & & 8.6395e-005 & \\
\hline
$4 $ & 8.1524e-006 & 2.9056& 5.8490e-005 & 2.0628 & 1.4715e-005 &2.5536\\
\hline
$8$ & 5.1367e-007& 3.9883& 4.2611e-006 & 3.7789 & 8.8043e-007 &4.0629\\
\hline
$16$ & 3.1999e-008 & 4.0047 & 2.7633e-007 & 3.9468 & 5.4626e-008 &4.0106\\
\hline
$32$ & 1.9980e-009 & 4.0013 & 1.7433e-008 & 3.9865 & 3.4292e-009 &3.9936\\
\hline
$64$ & 1.2489e-010 & 3.9999 & 1.0922e-009 & 3.9965 & 2.1480e-010 &3.9968 \\
\hline
\end{tabular}
\end{center}
\end{table}

\begin{table}[h!]
\begin{center}
\caption{ Numerical error and convergence rates in the $L^2$ norm, $H^1$ semi-norm and $L^\infty$ norm with exact solution $u=\cos(x)\cos(y)$ and coefficient tensor $A=[2,3;3,5]$.}\label{a6} \begin{tabular}{|c|c|c|c|c|c|c|}
\hline
$1/h$ & $\|u_h-u_I\|_0$ & order & $\|\nabla(u_h-u_I)\|_0$ &order & $\|u_h-u_I\|_{\infty}$ & order\\
\hline
$2 $ & 6.1392e-005 & & 2.4557e-004 & &
 8.6821e-005 & \\
\hline
$4$ & 8.1616e-006 &2.9111 & 5.8523e-005 &2.0690 & 1.4733e-005 & 2.5590\\
\hline
$8$ & 5.1423e-007 & 3.9884& 4.2631e-006 &3.7790 & 8.8150e-007 &4.0629\\
\hline
$16$ & 3.2034e-008 & 4.0047& 2.7645e-007 & 3.9468& 5.4693e-008 &4.0105\\
\hline
$32$& 2.0002e-009 & 4.0014& 1.7440e-008 & 3.9865& 3.4338e-009 &3.9935\\
\hline
$64$ & 1.2487e-010 & 4.0016 &1.0920e-009 &3.9973 & 2.1480e-010 &3.9988 \\
\hline
\end{tabular}
\end{center}
\end{table}

\bigskip

\section{Appendix}

In this section, we shall derive the Euler-MacLaurin formula that plays a crucial role in the superconvergence analysis for the finite element solution of the second order elliptic problem. As the Euler-MacLaurin formula can be found in most standard textbooks, the presentation of this formula is merely for self-completeness of the mathematical analysis.

\begin{lemma}\label{eul} (Euler-MacLaurin Formula)
Assume that $f(x)$ is sufficiently regular satisfying $f(x)\in H^4([a,b])$. There holds
\begin{equation}\label{ML}
\int_a^b f(x)dx=\int_a^b f_I(x)dx - \frac{(b-a)^2}{12} \int_a^b f''(x)dx+{\cal R}(x),
\end{equation}
where $f_I(x)$ is the linear interpolation of $f(x)$ on the interval $[a,b]$ given by $f_I(x)=\frac{b-x}{b-a}f(a)+\frac{x-a}{b-a}f(b)$, ${\cal R}(x)$ is the reminder term given by
$$
{\cal R}(x)=\left(\frac{b-a}{2}\right)^4 \int_a^b {\cal G}(x) f^{(4)}(x)dx
$$
with the weight function ${\cal G}(x)=\frac{1}{4!}\Big( \left(\frac{2}{b-a}\right)^2\left(x-\frac{a+b}{2}\right)^2-1\Big)^2$.
\end{lemma}

\begin{proof}
It suffices to derive the Euler-MacLaurin formula on the reference interval $[a,b]=[-1,1]$. To this end, we use the usual integration by parts to obtain
\begin{equation*}
\begin{split}
\int_{-1}^1 (f(t)- f_I(t))dt =& -\int_{-1}^1 f'(t)t dt\\
=& -\frac{1}{2}\int_{-1}^1 f'(t)d(t^2-1)
= \frac{1}{2}\int_{-1}^1  f''(t)  (t^2-1) dt  \\
=& - \frac{1}{3} \int_{-1}^1  f''(t) dt+ \frac{1}{3!}  \int_{-1}^1 f''(t)d (t(t^2-1)) \\
=& - \frac{1}{3} \int_{-1}^1  f''(t) dt- \frac{1}{3!}  \int_{-1}^1 f'''(t)t(t^2-1)dt \\
=& - \frac{1}{3} \int_{-1}^1  f''(t) dt- \frac{1}{4!}  \int_{-1}^1 f'''(t)  d(t^2-1)^2 \\
=&  - \frac{1}{3} \int_{-1}^1  f''(t) dt+\frac{1}{4!}  \int_{-1}^1 (t^2-1)^2f^{(4)}(t)  dt.\\
\end{split}
\end{equation*}
The Euler-MacLaurin formula on the general interval $[a,b]$ can now be obtained through
the transformation $x=h(t)=\frac{a+b}{2} +\frac{b-a}{2} t$ and the above expansion. Details are left to interested readers as an exercise. This completes the proof of the lemma.
\end{proof}


\begin{thebibliography}{00}
%% \bibitem{label}
%% Text of bibliographic item
\bibitem{bs1997} {\sc
J. H. Bramble and A. H. Schatz}, {\em Higher order local accuracy by averaging in the finite
element method}, Math. Comp., 31 (1977), pp. 94-111.

\bibitem{cz2015}{\sc
W. Cao,  Z. Zhang and Q. Zou}, {\em
Is 2K-conjecture valid for finite volume methods?  }
 SIAM J. Numer.Anal.,
53(2) (2015), pp. 942-962.

\bibitem{c2015}{\sc
W. Cao, C. Shu, Y. Yang and Z. Zhang}, {\em
Superconvergence of discontinuous Galerkin methods for two-dimensional hyperbolic equations},
 SIAM J. Numer.Anal.,
53(4) (2015), pp. 1651-1671.

\bibitem{cw2003}{\sc
H. Chen and J. Wang}, {\em An interior estimate of superconvergence for finite element solutions for second-order elliptic problems on quasi-uniform meshes by local projections}, SIAM J. Numer.Anal.,
41(4) (2003), pp. 1318-1338.

\bibitem{ciarlet-fem}
{\sc P.G. Ciarlet}, \textit{The Finite Element Method for Elliptic
Problems}, Classics Appl. Math. 40, SIAM, Philadelphia, 2002.


\bibitem{dd1973}{\sc
J. Douglas and T. Dupont}, {\em Superconvergence for Galerkin methods for the two-point boundary problem via local projections}, Numer. Math., 21 (1973), pp. 270-278.

\bibitem{elw1991}{\sc
R. E. Ewing, R. D. Lazarov and J. Wang}, {\em Superconvergence of the velocity along the Gauss lines in mixed finite element methods}, SIAM J. Numer. Anal., 28 (1991), pp. 1015-1029.

\bibitem{Gilbarg-Trudinger}
{\sc David Gilbarg and Neil S. Trudinger}. {\em Elliptic Partial
Differential Equations of Second Order}. Springer-Verlag, Berlin,
second edition, 1983.

\bibitem{gr}
{\sc V. Girault and P. A. Raviart}, {\em Finite Element Methods
for the Navier-Stokes Equations: Theory and Algorithms},
Springer-Verlag, Berlin, 1986.

\bibitem{grisvard}
{\sc P. Grisvard}, {\em Elliptic Problems in Nonsmooth Domains},
Classics Appl. Math. 69, SIAM, Philadelphia, 2011.


\bibitem{k2005}{\sc
M. Krizek}, {\em Superconvergence phenomenon on three-dimensional meshes},
International Journal of Numerical Analysis and Modeling, 2(1) (2005). pp. 43-56.

\bibitem{kn1987}{\sc
M. Krizek and P. Neittaanmaki}, {\em On superconvergence techniques}, Acta Appl. Math., 9
(1987), pp. 175-198.

\bibitem{kn1998}{\sc
M. Krizek and P. Neittaanmaki}, {\em Bibliography on superconvergence. In
Proc. Conf. Finite Element Methods: Superconvergence}, Post-processing
and A Posteriori Estimates, pp. 315-348, New York, 1998. Marcel Dekker.


\bibitem{or1969}{\sc
L. A. Oganesjan and L. A. Ruhovec}, {\em An investigation of the rate of convergence of variational
difference schemes for second order elliptic equations in a two-dimensional region with smooth boundary}, Z. Vycisl. Mat. i Mat. Fiz., 9 (1969), pp. 1102-1120.

\bibitem{ssw1996}{\sc
A. H. Schatz, I. H. Sloan and L. B. Wahlbin}, {\em Superconvergence in finite element methods
and meshes that are symmetric with respect to a point}, SIAM J. Numer. Anal., 33(1996),
pp. 505-521.

\bibitem{w1991}{\sc
J. Wang}, {\em Superconvergence and extrapolation for mixed finite element methods on rectangular
domains}, Math. Comp., 56 (1991), pp. 477-503.

\bibitem{w1995}{\sc
L. B. Wahlbin}, {\em Superconvergence in Galerkin Finite Element Methods}, Lecture Notes in Math.
1605, Springer-Verlag, New York, 1995.


\bibitem{w2000}{\sc
J. Wang}, {\em A superconvergence analysis for finite element solutions by the least-squares surface
fitting on irregular meshes for smooth problems}, J. Math. Study, 33 (2000), pp. 229-243.

\bibitem{z1978}{\sc
M. Zlamal}, {\em Superconvergence and reduced integration in the finite element method}, Math
Comp., 32 (1978), pp. 663-685.


\bibitem{zl1989}{\sc
Q. Zhu and Q. Lin}, {\em Superconvergence Theory of the Finite Element Methods}, Hunan Science
Press, Changsha, China, 1989.

\bibitem{zz1992}{\sc
O. C. Zienkiewicz and J. Z. Zhu}, {\em The superconvergence patch recovery and a posteriori error
estimates}, Part 2, Error estimates and adaptivity, Internat. J. Numer. Methods Engrg.,
33(1992), pp. 1365-1382.



\end{thebibliography}
\end{document}